\documentclass[10pt]{amsart}
\usepackage{amsrefs}
\usepackage[english]{babel}
\usepackage[utf8]{inputenc}
\usepackage{amsmath,relsize,hyperref}
\usepackage{graphicx}
\usepackage{amssymb}
\usepackage{amsthm}

\numberwithin{equation}{section}

\usepackage{tikz-cd}
\tikzset{
  symbol/.style={
    draw=none,
    every to/.append style={
      edge node={node [sloped, allow upside down, auto=false]{$#1$}}}
  }
}

\usepackage{mathrsfs}
\usepackage[colorinlistoftodos]{todonotes}
\usepackage{enumitem}
\usepackage{yfonts}
\usepackage{ dsfont }

\usetikzlibrary{matrix}
\usetikzlibrary{calc,intersections}
\usepackage{adjustbox}

\newtheorem{thm}{Theorem}[section]
\newtheorem{lem}[thm]{Lemma}
\newtheorem{prop}[thm]{Proposition}

\theoremstyle{definition}
\newtheorem{eg}[thm]{Example}
\newtheorem{rmk}[thm]{Remark}

\newcommand{\U}{{\mathcal U}}
\renewcommand{\H}{{\mathcal H}}

\newcommand{\ie}{\emph{i.e.} }

\newcommand{\Hbar}{\overline{{\mathcal H}}}
\newcommand{\Mbar}{\overline{{\mathcal M}}}
\newcommand{\Aff}{{\mathbb A}}
\newcommand{\HH}{\mathbb {H}}

\newcommand{\Z}{\mathbb{Z}}
\newcommand{\Q}{\mathbb{Q}}
\newcommand{\C}{\mathbb{C}}
\newcommand{\Pro}{\mathbb{P}}
\renewcommand{\P}{\mathbb{P}}

\DeclareMathOperator{\Adm}{Adm}

\DeclareMathOperator{\id}{id}

\DeclareMathOperator{\Pic}{Pic}
\DeclareMathOperator{\Cl}{Cl}

\newcommand{\coloneq}{:=}

\begin{document}

\title[Divisors on stacks of hyperelliptic curves]{Chow classes of divisors on stacks of pointed hyperelliptic curves}
\author{Dan Edidin}
\author{Zhengning Hu}
\address{Department of Mathematics, University of Missouri, Columbia, MO 65211}
\email{edidind@missouri.edu}
\email{zhengning.hu@mail.missouri.edu}

\begin{abstract}
  We calculate the Chow classes of the universal hyperelliptic Weierstrass divisor
  $\Hbar_{g,w}$ and the universal $g^1_2$ divisor $\Hbar_{g,g^1_2}$.
  Our results are expressed in terms of a basis for $\Cl(\Hbar_{g,1})$
  and $\Cl(\Hbar_{g,2})$ computed by Scavia \cite{Sca:20}.
  \end{abstract}

\maketitle

\section{Introduction}
A  problem inspired by Harris and Mumford's paper \cite{HaMu:82} on
the Kodaira dimension of the moduli space of curves is to compute the Picard
classes
of naturally occurring effective divisors in $\Mbar_g$ and $\Mbar_{g,n}$.
Examples include divisors parametrizing $k$-gonal curves \cite{HaMu:82},
divisors parametrizing curves with an exceptional Weierstrass point
\cite{Dia:85},
the Weierstrass divisor on the universal stable curve $\Mbar_{g,1}$ \cite{Cuk:89}, and the divisor parametrizing curves with a linear series with Brill-Noether
number $-1$ \cite{EiHa:87a, EiHa:87b}. In this paper, we turn
our attention to stacks of stable hyperelliptic curves.
Since every hyperelliptic curve
has a unique $g^1_2$, we cannot use exceptional linear series to impose
conditions on hyperelliptic curves.
However, the universal stable hyperelliptic curve
$\Hbar_{g,1}$ has a Weierstrass divisor $\Hbar_{g,w}$ and the stack
$\Hbar_{g,2}$ contains the universal $g^1_2$, $\Hbar_{g,g^1_2}$, parametrizing
stable hyperelliptic curves $(C,p_1,p_2)$ where $p_1$ and $p_2$ sum to the
$g^1_2$.
The purpose of this paper is to compute the classes of these naturally occurring divisors in the divisor class groups
of $\Hbar_{g,1}$
and $\Hbar_{g,2}$. We express our results in terms of a basis
for $\Cl(\Hbar_{g,n})$ recently computed by Scavia \cite{Sca:20}.

\begin{thm}\label{main-1}
   \begin{eqnarray*} [\Hbar_{g,w}] & = & \left({g+1\over{g-1}}\right)\psi - {1\over{2(2g+1)(g-1)}}\eta_{irr} +\\
  & & \sum_{i =1}^{\lfloor (g - 1)/2 \rfloor} \left[-{(i+1)(2i+1)\over{(2g+1)(g-1)}} \eta_{i,0} -
      {(g-i)[2(g-i)-1]\over{(2g+1)(g-1)}} \eta_{i,1}\right] +\\
 & & \sum_{i=1}^{\lfloor g/2 \rfloor} \left[-{2i(2i+1)\over{(2g+1)(g-1)}} \delta_{i,0} -
      {2(g-i)[2(g-i) + 1]\over{(2g+1)(g-1)}}\delta_{i,1}\right].
    \end{eqnarray*}
\end{thm}

\begin{thm}\label{main-2}
  \begin{eqnarray*} [\Hbar_{g,g^1_2}]  & = &
    \left({1\over{g-1}}\right)(\psi_1 + \psi_2) - {1\over{2(g - 1)(2g + 1)}}\eta_{irr} - \left({g + 1\over{g - 1}}\right)\delta_{0,2} + \\
& &    \sum_{i =1}^{\lfloor (g - 1)/2 \rfloor} \left[-{(i+1)(2i+1)\over{(g-1)(2g+1)}} \eta_{i,0} +
  	{2i(g - i - 1) - 1\over{(g - 1)(2g + 1)}} \eta_{i,1} -
        {(g-i)[2(g-i)-1]\over{(g - 1)(2g+1)}} \eta_{i,2}\right] + \\
& &      \sum_{i=1}^{\lfloor g/2 \rfloor} \left[-{2i(2i+1)\over{(g-1)(2g+1)}} \delta_{i,0} +
 	{(2i - 1)[2(g - i) - 1] - 2\over{(g - 1)(2g + 1)}} \delta_{i,1} -
        {2(g-i)[2(g-i) + 1]\over{(g - 1)(2g+1)}}\delta_{i,2}\right]. 
    \end{eqnarray*}
\end{thm}

Here $\psi,\psi_1, \psi_2$ are the $\psi$-classes associated to the sections
and the other classes are boundary divisors which we describe below.
When $g = 2$, our formula in Theorem \ref{main-1} agrees with the
formula proved by Eisenbud and Harris in \cite{EiHa:87a} using
Porteous's formula. Unfortunately, we cannot apply their method in
higher genus, because 
the corresponding degeneracy locus has expected codimension $g -
1$. Instead we use the method of test
curves, and the main challenge is to find sufficiently interesting
families of branched double covers of rational nodal curves. The key
construction to do this is done in Section \ref{sec.general}.

{\bf Conventions and notation.} Throughout this paper, we fix a natural number $g \geq 2$. Because Scavia uses results from topology to compute the divisor class group of $\Hbar_{g,n}$ we work over the field $\C$ of complex numbers.

\section*{Acknowledgement}
The first author was supported by Simons Collaboration Grants 315460 and 708560.
The authors are very grateful to the referees for a careful reading and a number of helpful comments which improved the exposition.

\section{Stacks of hyperelliptic curves}
Let ${\mathcal H}_g$ be the stack of smooth hyperelliptic curves of genus $g \geq 2$.
This is a closed smooth substack of the moduli stack
${\mathcal M}_g$ and we denote by $\Hbar_g$ its closure 
in $\overline{{\mathcal M}}_g$. This is a smooth substack of $\Mbar_g$
by \cite[Chapter XI, Lemma (6.15)]{MR2807457}. As noted there, the
stack $\Hbar_g$ parametrizes curves which have a (unique) involution with isolated fixed points, whose quotient is a nodal rational curve.

Likewise, let $\Hbar_{g,n}$ denote the stack of $n$-pointed
hyperelliptic curves. It is defined as the stack-theoretic inverse image of $\Hbar_g$
under the forgetful map $\Mbar_{g,n} \to \Mbar_{g}$; i.e. we set
$\Hbar_{g,n}$ to equal the fiber product $\Hbar_{g} \times_{\Mbar_{g}} \Mbar_{g,n}$. (Our definition of $\Hbar_{g,n}$ differs slightly from Scavia's, in that Scavia takes the reduced substack structure -- an operation which does not change the divisor class group.)

In \cite{barros2021kodaira}, the authors demonstrate that
$\Hbar_{g,n}$ is not smooth when $g \geq 3$ and $n \geq 2$, but 
Scavia \cite{Sca:20} proves that $\Hbar_{g,n}$ always contains a smooth open substack $\U_{g,n}$
whose complement has codimension at least two. The stack $\U_{g,n}$
is the union of two open substacks $\H^{rt}_{g,n}$ and $\Hbar^o_{g,n}$. The
first open substack is the inverse image of $\H_g$ under the projection
$\Hbar_{g,n} \to \Hbar_g$ and parametrizes curves with rational
tails which contract to a smooth curve.
The second open substack parametrizes stable pointed curves which
remain stable after deleting the marked points.

\subsection{Stable hyperelliptic curves as admissible covers}
In order to study the boundary of $\Hbar_g$ we
use the fact that any stable hyperelliptic curve can be obtained by stabilizing
an admissible cover. 
Once we do this we will interpret the divisors
$\Hbar_{g,w}$ and $\Hbar_{g,g^1_2}$ in the context of admissible covers.
This will allow us to determine
their intersection with the boundary
divisors of $\Hbar_{g,1}$ and $\Hbar_{g,2}$ respectively.

Let
$\Adm_{0,2g+2,2}$ be the stack whose sections over a base scheme $S$ parametrize
the following data:

A stable pointed rational curve $(P \to S,\Sigma_{2g+2})$
and a
double cover $C \to P$ which is \'etale over the complement
of $\Sigma_{2g+2}$ and the nodes of $P$, and satisfies
\begin{enumerate}
\item $C \to S$ is a nodal curve.
\item Every node of $C$ maps to a node of $P$, and over the nodes
  of $P$ the map $C \to P$ can be described as in \cite[Definition 4.1.1]{ACV:03}. In particular they may be either branched or \'etale over the nodes.
  \item $C \to P$ is ramified over the divisor $\Sigma_{2g+2}$.
\end{enumerate}

\begin{rmk}
  Note that each fiber of $C \to S$ is a stable curve unless the corresponding
  fiber of $P \to S$
contains a tail with exactly two markings coming from $\Sigma_{2g+2}$. In this case it is the semi-stable curve obtained by replacing the node of an irreducible component with a rational bridge.
\end{rmk}

By \cite{ACV:03}, the stack $\Adm_{0,2g+2,2}$ is a proper Deligne–Mumford stack. In general, stacks of admissible covers are not normal, but because double covers are always Galois, this stack is isomorphic to the stack of twisted stable maps to $B\mu_2$ which is
smooth \cite[Theorem 3.0.2]{ACV:03}.

The map $\Adm_{0,2g+2,2} \to \Mbar_g$, which
associates to an admissible double cover the corresponding stable curve
of genus $g$, is a proper morphism with $0$-dimensional fibers and 
the substack $\Hbar_{g}$ is the image of $\Adm_{0,2g+2,2}$. Note that there is
a natural free action of the symmetric group on $\Adm_{0,2g+2,2}$ and
the map $\Adm_{0,2g+2,2} \to \Mbar_g$ factors through the quotient by $[\Adm_{0,2g+2,2}/S_{2g+2}]$ in the sense of \cite{Rom:05}.

Thus we obtain a proper surjective morphism $[\Adm_{0,2g+2,2}/S_{2g+2}] \to \Hbar_{g}$ with $0$-dimensional
fibers which is a bijection on $\C$-valued points. However, the following example, which we learned from Andrea di Lorenzo, shows that the morphism $[\Adm_{0,2g+2,2}/S_{2g+2}] \to \Hbar_{g}$ is not representable, and therefore it is not an equivalence.

\begin{eg}[Andrea Di Lorenzo]
  Consider the admissible cover of the rational curve $P$ which is the union
  of two copies of $\P^1$. Mark $2g$ general points on one component and two points
  on the other component. Let $C' \to P$ be the admissible double cover of $P$
  branched at the marked points. The curve $C'$ is the union of a general hyperelliptic curve $D$ of genus $g-1$ and a $\P^1$ which meets $D$ in two points.
  The automorphism group of the admissible cover $C' \to \Pro^1$ contains two non-trivial elements. The first
  is the automorphism which restricts to the hyperelliptic involution on
  $D$ and the automorphism of $\P^1$ which fixes the two ramification
  points of the double cover $\P^1 \to \P^1$,
  and exchanges the two points of attachment of the $\P^1$ to $D$.

  The other automorphism restricts to the identity on $D$, and on the $\P^1$ fixes the points of attachment but exchanges the two ramification points. See Figure \ref{nontrivialauto}.
  
\begin{figure}[h!]\centering
\begin{tikzpicture}[extended line/.style={shorten >=-#1,shorten <=-#1},
 extended line/.default=1cm,
 one end extended/.style={shorten >=-#1},
 one end extended/.default=1cm,]
\draw [name path = curve1,ultra thick] 
	(1,1) to[out=190,in=110] (-0.8,-0.5) to[out=-65,in=220] (1,0);
\draw [name path = line1,ultra thick]
	(-1,-2) -- (1,-1);

\draw [name path = curve2,ultra thick] 
	(0,1) to[out=-10,in=70] (1.3,-0.2) to[out=-115,in=-40] (0,0);
\draw [name path = line2,one end extended = 0.4cm,ultra thick]
	(0,-1) -- (2,-2);

\draw[rotate around={150:(1.6,-0.4)},ultra thick] (1.6,-0.4) ellipse (0.36 and 0.3);

\node[fill,circle,inner sep=2pt] at (1.3,-0.2) {};
\node[fill,circle,inner sep=2pt] at (1.9,-0.56) {};

\path[name path = ver1] (1.3,-0.2) -- (1.3,-2);
\path[name intersections={of=ver1 and line2, by={a}}];
\node[fill,circle,inner sep=2pt] at (a) {};

\path[name path = ver2] (1.9,-0.56) -- (1.9,-2);
\path[name intersections={of=ver2 and line2, by={b}}];
\node[fill,circle,inner sep=2pt] at (b) {};

\node (A) at (-0.2,0.4) {};
\node (B) at (-0.2,-0.6) {};

\draw[<->]
	(A) edge (B);

\node (C) at (0.5,-0.1) {};
\node (D) at (0.5,0.7) {};

\draw[<->]
	(C) edge (D);
	
\node at (-1,0.5) {$D$};
\node at (2,-1) {$\Pro^1$};

\node at (1.7,0.2) {$\id$};

\begin{scope}[xshift=2cm]
\draw [name path = curve3,ultra thick] 
	(3,1) to[out=190,in=110] (1.2,-0.5) to[out=-65,in=220] (3,0);
\draw [name path = line3,ultra thick]
	(1,-2) -- (3,-1);
\draw [name path = curve4,ultra thick] 
	(2,1) to[out=-10,in=70] (3.3,-0.2) to[out=-115,in=-40] (2,0);
\draw [name path = line4,one end extended = 0.4cm,ultra thick]
	(2,-1) -- (4,-2);
	
\draw[rotate around={150:(3.6,-0.4)},ultra thick] (3.6,-0.4) ellipse (0.36 and 0.3);

\node[fill,circle,inner sep=2pt] at (3.3,-0.2) {};
\node[fill,circle,inner sep=2pt] at (3.9,-0.56) {};

\path[name path = ver3] (3.3,-0.2) -- (3.3,-2);
\path[name intersections={of=ver3 and line4, by={c}}];
\node[fill,circle,inner sep=2pt] at (c) {};

\path[name path = ver4] (3.9,-0.56) -- (3.9,-2);
\path[name intersections={of=ver4 and line4, by={d}}];
\node[fill,circle,inner sep=2pt] at (d) {};

\node (E) at (1.7,0) {$\id$};
\node (F) at (2.5,0.3) {$\id$};

\node (G) at (3.32,-0.15) {};
\node (H) at (3.92,-0.55) {};

\draw[<->]
	(G) edge[bend left=90] (H);
	
\node at (1,0.5) {$D$};
\node at (4,-1) {$\Pro^1$};
\end{scope}

\begin{scope}[xshift=4cm]
\node (X) at (3.5,0) {$C'$};
\node (Y) at (3.5,-2) {$P$};
\draw[->] (X) edge (Y);
\end{scope}

\node at (0,-2) {};
\end{tikzpicture}

\caption{Two non--trivial automorphisms}
\label{nontrivialauto}
\end{figure}
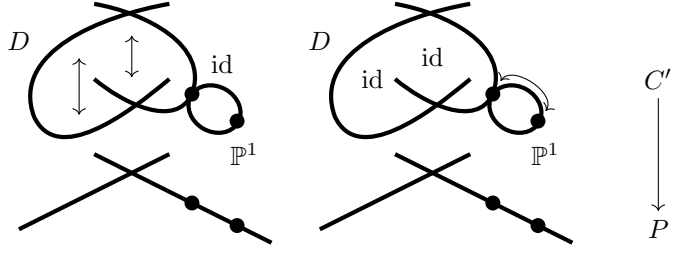

  On the other hand, the stabilization of $C'$ is an irreducible nodal hyperelliptic curve. This curve has only one non-trivial automorphism which is the hyperellipic involution.
\end{eg}

\begin{rmk} A section of the quotient $[\Adm_{0,2g+2,2}/S_{2g+2}]$ over a base scheme $S$ is given by the following data:
\begin{enumerate}
\item A rational curve $P \to S$ and a divisor $\Sigma_{2g+2} \subset P$ which
  is finite and \'etale over $S$ and such that the geometric fibers of $P \to S$
are stable when marked by the pullback of $\Sigma_{2g+2}$ to the fiber.
\item An admissible double cover $C \to P$ which is ramified over
  $\Sigma_{2g+2}$.
\end{enumerate}
\end{rmk}

\subsection{The Weierstrass divisor in $\Hbar_{g,1}$}
Let $\H_{g,w}$ be the substack of $\H_{g,1}$ parametrizing
Weierstrass points in the fibers of the projection $\H_{g,1} \to
\H_g$. By \cite[Corollary 6.8]{KlLo:79}, $\H_{g,w}$ is \'etale of
degree $2g+2$ over $\H_{g}$ and hence $\H_{g,w}$ is a smooth divisor
in $\H_{g,1}$.  The following result is certainly well known but
lacking a suitable reference, and we include a (sketch of a) proof.
\begin{prop}\label{prop.irr}
  $\H_{g,w}$ is irreducible.
\end{prop}
\begin{proof}
  We use a construction of $\H_{g,1}$ as a quotient stack given by
  Pernice in \cite{pernice2020integral}. Specifically he proves that $\H_{g,1}$ is
  a quotient of the scheme $\widetilde{\Aff}_{sm}(2g+2) \subset
  \Aff^{2g+3} \times \Aff^1$ by a connected algebraic group $B$. The
  scheme $\widetilde{\Aff}_{sm}(2g+2)$ parametrizes pairs $(f,s)$ where
  $f$ is a binary form of degree $2g+2$ with simple roots and $f(0:1)
  = s^2$. With this description the Weierstrass divisor is the
  quotient of the irreducible closed subscheme defined by setting
  $s=0$. Hence $\H_{g,w}$ is irreducible since it has a smooth cover
  by an irreducible scheme.
\end{proof}

We denote by $\Hbar_{g,w}$ the closure (with its canonical closed substack structure)
in $\Hbar_{g,1}$ of the divisor
$\H_{g,w}$. Since $\H_{g,w}$ is irreducible so is $\Hbar_{g,w}$. The discussion
after \cite[Proposition 1]{Cor:07} shows that $\Hbar_{g,w}$ is a Cartier divisor
in $\Hbar_{g,1}$. Precisely if $C \to S$ is a family of stable hyperelliptic curves with involution $\tau$, then the restriction of $\Hbar_{g,w}$ to $C$ is the fixed locus of $\tau$ minus the nodes of type $\Delta_i$ for $i > 0$. In particular if $(C \to S, \sigma)$ is a stable pointed family of hyperelliptic curves such that $C \to S$ is still stable after forgetting the section (i.e. the fibers of $C \to S$ do not contain rational bridges) then the pullback of
$\Hbar_{g,w}$ along the morphism $S \to \Hbar_{g,1}$ is the divisor $\sigma^* \tau^* {\mathcal O}(\sigma)$.

\subsection{The $g^1_2$ divisor in $\Hbar_{g,2}$}
Let $\H_{g,g^1_2} \subset \H_{g,2}$ be the divisor parametrizing
two-pointed hyperelliptic curves $(C,p_1,p_2)$ such that ${\mathcal O}(p_1 + p_2)$ is the $g^1_2$ on $C$.

Note that the substack $\H_{g,g^1_2}$ is representable since the hyperelliptic involution cannot fix the $g^1_2$ divisor. It is equivalent to the representable open substack  $\H_{g,1}\setminus \H_{g,w}$ since for any non--Weierstrass point $p$ on a hyperelliptic curve there is a unique point $q$ such that $p+q$ is a $g^1_2$. In particular, $\H_{g,g^1_2}$ is irreducible. Let $\Hbar_{g,g^1_2}$ be its closure in $\Hbar_{g,2}$ with its reduced induced substack structure.

Although we do not know if $\Hbar_{g,g^1_2}$ is a Cartier divisor, we will use the following explicit description of $\Hbar_{g,g^1_2}$ as a Cartier divisor on the smooth open substack $\H^{rt}_{g,2}$ parametrizing two-pointed curves with at most one rational tail and which stabilize to a smooth curve.

The morphism $\H^{rt}_{g,2} \to \H_{g,1}$ which forgets the first section makes $\H^{rt}_{g,2}$ into the universal stable pointed hyperelliptic curve
over $\H_{g,1}$. The fibers of $\H^{rt}_{g,2} \to \H_{g,1}$ are one-pointed
smooth curves and there is a hyperelliptic involution $\tau \colon \H^{rt}_{g,2}\to \H^{rt}_{g,1}$ which commutes with the projection to $\H_{g,1}$. Identify the open set $\H_{g,2} \subset \H^{rt}_{g,2}$ as parametrizing pairs $((C,p),q)$ where
$(C,p)$ is a marked curve (\ie an object of $\H_{g,1}$) and $q$ is a point
on $C$ distinct from $p$. With our definition, $((C,p), q)$ is in $\H_{g,g^1_2}$ if and only if $q = \tau(p)$. In other words, we identify
$\H_{g,g^1_2}$ with the intersection of the divisor $\tau(\sigma)$ with the open
set $\H_{g,2} \subset \H_{g,2}^{rt}$. Thus the closure of $\H_{g,g^1_2}$ in
$\H_{g,2}^{rt}$ is the Cartier divisor $\tau(\sigma)$ where $\sigma$ is the universal section.

This observation can be applied to compute the degree $\Hbar_{g,g^1_2}$ on
a family $(C \to S, \sigma_1, \sigma_2)$ where $S$ is a smooth projective
which maps to $\H_{g,2}^{rt}$.
\begin{prop}
  Let $(C \to S, \sigma_1, \sigma_2)$ be a family as above and let
  $(C' \to S, \sigma)$ be the family obtained by deleting the section
  $\sigma_2$ and stabilizing. If $f \colon C \to C'$ is the stabilization
  map, then $\deg_S [\Hbar_{g,g^1_2}] = \deg \sigma_2^*f^*(\tau(\sigma))$
  where $\tau$ is the hyperelliptic involution on $S$.
\end{prop}
\begin{proof} \label{prop.hg12.normal}
  The map $S \to \H_{g,2}^{rt}$ factors as $S  \stackrel{\sigma_2} \to  C \stackrel{f} \to C' \to \H_{g,2}^{rt}$.
  By the discussion above, the pullback of $\Hbar_{g,g^1_2}$ to $C'$ is the Cartier divisor $\tau(\sigma)$.
\end{proof}

\section{Divisor class groups of stacks of hyperelliptic curves}
  In this Section we recall from \cite[Section 13.8]{MR2807457} and \cite{Sca:20}
  the description of the divisor class groups of $\Hbar_g, \Hbar_{g,1}, \Hbar_{g,2}$.
  
\subsection{The divisor class group of $\Hbar_g$}
The theory of admissible covers allows us
to describe the boundary divisors of $\Hbar_g$; \ie the irreducible components of $\Hbar_g \setminus {\mathcal H}_g$. See \cite[Section 13.8]{MR2807457} for a reference.

\begin{itemize}
\item $\eta_{irr}$: This parametrizes stable hyperelliptic curves $C$ with at least one non-separating node such that its partial normalization at such point does not have a separating node. A general curve in $\eta_{irr}$ has a single node. Its normalization
  is a smooth hyperelliptic curve of genus $g-1$. If $\{p, q\}$ is
  the inverse image of the node, then $p+q$ equals the $g^1_2$. (If $g=2$ then
  the normalization has genus $1$ and there
  is no condition on the points since any two points on an elliptic curve sum to a
  $g^1_2$.)
  
\item $\delta_i$ for $1 \leq i \leq \lfloor\frac{g}{2}\rfloor$: For each
  $i$ this parametrizes curves $C$ with a disconnecting node
such that the partial normalizations at the node is the disjoint union of a curve of genus $i$ and a curve of genus $g - i$. Moreover, such a node is fixed by the hyperelliptic involution of $C$, which means that
  each point in the inverse image of the node is a \textit{Weierstrass point}
  on its corresponding component.
  
\item $\eta_i$ for $1 \leq i \leq \lfloor\frac{g-1}{2}\rfloor$: They
  parametrize curves $C$ having two nodes that are conjugated by the hyperelliptic involution such that the partial normalization is the disjoint union of two curves of genera $i$ and $g - 1 - i$. The points of attachment on each component sum to a $g^1_2$.
\end{itemize}

\begin{thm}  \cite[Chapter XIII, Theorem (8.4)]{MR2807457}, \cite{Cor:07}
  For any $g \geq 2$, $\Pic(\Hbar_g)_\Q$ is a vector space of dimension $g$, freely generated by the classes $\eta_{irr}$, $\{\delta_i\}_{1 \leq i \leq
    \lfloor g/2 \rfloor}$, and $\{\eta_i\}_{1 \leq i \leq \lfloor(g-1)/2\rfloor}$.
\end{thm}

\subsection{The divisor class groups of $\Hbar_{g,1}$ and $\Hbar_{g,2}$ }
We
begin by describing the inverse images of these boundary divisors under
the projection $\pi \colon \Hbar_{g,1} \to \Hbar_g$.

\begin{itemize}
\item $\eta_{irr}$: The inverse
  image of $\eta_{irr}$ is an irreducible component of $\Hbar_{g,1} \setminus
  \H_{g,1}$ and we again denote it by $\eta_{irr}$.
  
\item $\delta_i$: If $i < g/2$ then the inverse image of $\delta_{i}$
  consists of two irreducible components $\delta_{i,1}$ and $\delta_{i,0}$
  where the second index is one if the marked point is on the component
  of genus $i$ and zero if the marked point is on the component of
  genus $g-i$. If $i = \frac{g}{2}$ then the inverse image of $\delta_i$
  is irreducible.

\item $\eta_i$: If $i <  \frac{g-1}{2}$ then the inverse
  image of $\eta_i$ again has two irreducible components
  $\eta_{i,0}$ and $\eta_{i,1}$ corresponding to whether the marked point
  is on the component of $g-i-1$ or the component of genus $i$. When
  $i = \frac{g-1}{2}$, then $\eta_{i,0} = \eta_{i,1}$.
  \end{itemize}
\begin{rmk}
  Note that the divisors $\eta_{irr}$, $\delta_{i,0}, \delta_{i,1}$ are Cartier divisors because they can be identified with pullbacks of divisors on the smooth
  stacks $\Hbar_g$ and $\Mbar_{g,1}$ respectively. For $1 \leq i < g/2$
  the sum $\eta_{i,0} + \eta_{i,1}$ is Cartier because it is the pullback of the Cartier divisor $\eta_i$ on $\Hbar_{g}$. However, we do not know if
  each component is Cartier. A potential way to show that $\Hbar_{g,1}$
  is singular is to show that the divisors $\eta_{i,0}$ and $\eta_{i,1}$
  are not locally defined by a single equation along their intersection.
\end{rmk}

The boundary divisors of $\Hbar_{g,2}$ are defined in a similar 
manner. Again the inverse image of $\eta_{irr}$ is irreducible but we need to keep track of the markings for the other components.
\begin{itemize}
\item $\delta_i$: If $i < g/2$ then the inverse image of $\delta_{i}$
  consists of four irreducible components $\delta_{i,0}$, $\delta_{i,2}$
  $\delta_{i,\{1\}}$ and $\delta_{i, \{2\}}$ defined as follows.
  The divisor $\delta_{i,0}$ corresponds to both marked points being on the component of genus $g-i$, $\delta_{i,2}$ corresponds to both marked points being on the component of genus $i$, and $\delta_{i,\{j\}}$ for $j =1,2$ means that
  the $j$th marked point is on the component of genus $i$. We denote
  by $\delta_{i,1}$ the sum $\delta_{i,\{1\}} + \delta_{i,\{2\}}$.
  The divisor $\delta_{i,1}$ is invariant under the involution of $\Hbar_{g,2}$
  which exchanges the marked points.
If $i = \frac{g}{2}$ then the inverse image of $\delta_i$
consists of two irreducible components $\delta_{i,0} = \delta_{i,2}$
and $\delta_{i,1} = \delta_{i, \{1\}} = \delta_{i, \{2\}}$. 

\item $\eta_i$: If $i <  \frac{g-1}{2}$ then the inverse
  image of $\eta_i$ again has four irreducible components
  $\eta_{i,0}$ and $\eta_{i,2}$, $\eta_{i,\{1\}}$, $\eta_{i,\{2\}}$
  corresponding to the distribution of the marked points
  on the component of genus $g-i-1$ or the component of genus $i$. Again we denote by
  $\eta_{i,1}$ the sum $\eta_{i,\{1\}} + \eta_{i,\{2\}}$ and the divisor
    $\eta_{i,1}$ is invariant under the involution of $\Hbar_{g,2}$. When
  $i = \frac{g-1}{2}$, then $\eta_{i,0} = \eta_{i,2}$ and $\eta_{i,1} = \eta_{i,\{1\}} = \eta_{i, \{2\}}$.

\item  $\delta_{0,2}$: The general curve is a smooth
hyperelliptic curve with a rational tail with the two marked points
lying on it. Moreover, since a rational curve with three distinct points does not move in moduli, in order to produce stratum $\delta_{0,2}$ with dimension $2g$, we impose no condition on the point on the genus $g$
component for attaching a rational tail with two marked points.
\end{itemize}
\begin{rmk}
  Again the divisors $\eta_{irr}, \delta_{i,j}$ are Cartier, but we do not know
  if the divisors $\eta_{i,0},\eta_{i,2},\eta_{i,\{1\}}$ and $\eta_{i,\{2\}}$ are Cartier along their intersections.
\end{rmk}

\begin{thm} \label{thm.scavia} \cite[Theorem 1.1]{Sca:20}
$\Cl(\Hbar_{g,n})_\Q = \Pic(\U_{g,n})_\Q$ 
  is freely generated by the
  classes $\psi_1, \ldots, \psi_n$, and all boundary divisors, where the
  $\psi_i$ are the pullbacks to $\Hbar_{g,n}$ of the corresponding
  $\psi$-classes\footnote{Recall that the class $\psi_i$ is the line bundle on the stack $\Mbar_{g,n}$
    whose pullback to a pointed family of stable curves
    $(C \stackrel{p} \to B,\sigma_1, \ldots , \sigma_n )$ is the line bundle $\sigma_i^*(\omega_{C/B})$.} on $\Mbar_{g,n}$.
\end{thm}

By Theorem \ref{thm.scavia} we can write 
\begin{align} \label{eq.hgw}
[\overline{\mathcal{H}}_{g,w}] = \sum_{i = 1}^{\lfloor g/2 \rfloor}[a_{i,0}\delta_{i,0} + a_{i,1}\delta_{i,1}] + \sum_{j = 1}^{\lfloor (g-1)/2 \rfloor}[b_{j,0}\eta_{j,0} + b_{j,1}\eta_{j,1}] + c\eta_{irr} + d\psi
\end{align}
for some rational numbers $a_{i,0}, a_{i,1}, b_{j,0}, b_{j,1}, c, d$
where $\psi = \psi_1 \in \Pic(\Hbar_{g,1})$.
The divisor $\Hbar_{g,g^1_2}$ is invariant under the involution
that exchanges the marked points, so it can be expressed as a sum of invariant divisors
\begin{align} \label{eq.hg12}
[\Hbar_{g,g^1_2}] & = c\eta_{irr} + d(\psi_1 + \psi_2) + a_{0,2}\delta_{0,2} + \nonumber\\
& \phantom{{=}} \sum_{i = 1}^{\lfloor 2/g \rfloor}\left[a_{i,0}\delta_{i,0} + a_{i,1}\delta_{i,1} + a_{i,2}\delta_{i,2}\right] +  
\sum_{j = 1}^{\lfloor (g-1)/2\rfloor}\left[b_{j,0}\eta_{j,0} + b_{j,1}\eta_{j,1} + b_{j,2}\eta_{j,2}\right]
\end{align}
where the $a_{i,0}, a_{i,1}, a_{i,2}, b_{j,0}, b_{j,1}, b_{j,2}, c, d$ are rational numbers, and are unrelated to those of \eqref{eq.hgw}.

Our goal is to determine all the $\Q$-coefficients of these generators. We will use the method of test curves as has been used widely in the literature.
To do this we will begin by determining how the divisors $\Hbar_{g,w}$
and $\Hbar_{g,g^1_2}$ intersect the boundary strata of $\Hbar_{g,1}$
and $\Hbar_{g,2}$ respectively. All of these assertions are easily checked using the the descriptions of the divisors $\Hbar_{g,w}$ and $\Hbar_{g,g^1_2}$
given in terms of admissible covers.

\subsection{Intersection of $\Hbar_{g,w}$ with the boundary of $\Hbar_{g,1}$}
We summarize our results as follows.
\begin{prop}\label{prop.intbdy.hgw}
  $\Hbar_{g,w}$ intersects each of the boundary divisors in codimension-one,
  and a general
  point of the intersection is described as follows.
  \begin{enumerate}
  \item The general point of the intersection $\Hbar_{g,w}$ with $\eta_{irr}$
    is an irreducible curve $C$ of arithmetic genus $g$ with a single node such that the marked point is a Weierstrass point of the normalization of $C$.

  \item The general point of the intersection with  $\delta_{1,1}$ is
    a general point of $\delta_{1} \in \Hbar_g$ with the marked
    point $p$ on the component $E$ of genus $1$ having the property that
    ${\mathcal O}(p-o)^2 = 0$ where $o \in E$ is the point of attachment of $E$.

  \item The general point of the intersection with $\delta_{i,1}$ for $i > 1$
    is a general point of $\delta_i$ with the marked point a Weierstrass point
    of the component of genus $i$.
    
  \item The general point of the intersection with $\delta_{i,0}$
    is a general point of $\delta_i$ with the marked point a Weierstrass point
    of the component of genus $g-i$.

  \item The general point of the intersection with $\eta_{1,1}$
    is a general point of $\eta_1$ with the marked point $p$ on
    the component of genus $1$ satisfying the condition ${\mathcal O}(2p) = {\mathcal O}(o_1 + o_2)$
    where $o_1, o_2$ are the points of attachment on the component of genus $1$.
    
  \item The general point of the intersection with $\eta_{i,1}$ with $i > 1$
    is a general point of $\eta_i$ with the marked point a Weierstrass point
    of the component of genus $i$.

  \item The general point of the intersection with $\eta_{i,0}$ is a general
    point of $\eta_i$ with the marked point being a Weierstrass point of
    the component of genus $g-i-1$.
  \end{enumerate}
\end{prop}
\begin{proof}
The proof can be illustrated by admissible covers. We consider
degenerations of a $2g + 2$-pointed rational curve as a stable $2g +
2$-pointed rational curve with a single node, with various distributions
of the $2g + 2$  points. When taking double
cover, the single node may or may not be branched, depending on the
parity of the number of markings on each component. For instance, if
there are two points on one twig and $2g$ on the other, then the
marked Weierstrass point must lie over one of the $2g$ points which
proves (1). For any $i > 1$, if there are $2i + 1$ points on one twig
and the rest $2g - 2i + 1$ on the other, then the single node is
branched and thus there are two choices for the marked Weierstrass
point, corresponding to $\delta_{i,0}$ and $\delta_{i,1}$. It proves
(3) and (4). The rest follows similarly.
\end{proof}

\subsection{Intersection of $\Hbar_{g,g^1_2}$ with the boundary of $\Hbar_{g,2}$}

\begin{prop}\label{prop.intbdy.hg12}
  $\Hbar_{g,g^1_2}$ intersects each of the boundary divisors in codimension one,
  and a general
  point of the intersection is described as follows. 
  \begin{enumerate}
  \item The general point of the intersection $\Hbar_{g,g^1_2}$ with $\eta_{irr}$
    is an irreducible curve $C$ of arithmetic genus $g$ with a single node such that the two marked points sum to a $g^1_2$ on the normalization of $C$.

  \item The general point of the intersection with $\delta_{0,2}$
    is a curve where the point of attachment on the curve of genus $g$
    is a Weierstrass point.

  \item The general point of the intersection with  $\delta_{1,2}$ is
    a general point of $\delta_{1} \in \Hbar_g$ with the marked
    points $p,q$ on the component of genus $1$ having the property that
    ${\mathcal O}(p+q) = {\mathcal O}(2o)$. where $o$ is the point of attachment of the component
    of genus $1$.

  \item The general point of the intersection with $\delta_{i,2}$ for $i > 1$
    is a general point of $\delta_i$ with the marked points $p,q$
    summing to a $g^1_2$ on the component of genus $i$.

  \item The general point of the intersection with $\delta_{i,0}$
    is a general point of $\delta_i$ with the marked points summing to 
   a $g^1_2$ on the component of genus $g-i$.

  \item The general point of the intersection with $\eta_{1,2}$
    is a general point of $\eta_1$ with the marked points $p,q$ on
    the component of genus $1$ satisfying the condition ${\mathcal O}(p+q) = {\mathcal O}(o_1 + o_2)$
    where $o_1, o_2$ are the points of attachment on the component of genus $1$.
    
  \item The general point of the intersection with $\eta_{i,2}$ for $i > 1$
    is a general point of $\eta_i$ with the marked points summing to a $g^1_2$
    on the component of genus $i$.

  \item The general point of the intersection with $\eta_{i,0}$ is a general
    point of $\eta_i$ with the marked points summing to a $g^1_2$ on
    the component of genus $g-i-1$.

    \end{enumerate}
\end{prop}
\begin{proof}
The proof is similar to that of Proposition \ref{prop.intbdy.hgw} except while taking the double cover, the pair of $g^1_2$ points lies over a point on a twig away from any of the $2g + 2$ markings.
\end{proof}

\section{Methods of constructing test curves}\label{sec.methods}
We describe four basic methods to construct complete one-parameter families
of curves in $\Hbar_{g,1}$, $\Hbar_{g,2}$. All of our test curves will be based on these constructions and will lie in the smooth locus of $\Hbar_{g,1}$
and $\Hbar_{g,2}$ respectively.

\subsection{Test curves with trivial moduli}\label{sec.trivial}
The simplest examples of one-parameter families of pointed hyperelliptic curves
are families $(X \to B,\sigma)$ where the section moves but the moduli of the curves does not vary. In this case the composite map $B
\to \Hbar_{g,1} \to \Hbar_{g}$ is constant.
The basic construction is as follows. Let $C$ be a fixed smooth hyperelliptic curve of genus $h$ and consider the family $C \times C \to C$ with the section
$\Delta_C \colon C \to C \times C$ being the diagonal.

Choosing $k$-points $p_1, \ldots , p_k$
on $C$ we can obtain a family of hyperelliptic curves with $k+1$ sections
$X \to C$ where $X$ is the surface obtained by blowing up $C \times C$
at the points $(p_j, p_j)$ for $j=1,\ldots k$. The sections $\sigma_0,\sigma_1, \ldots ,
\sigma_k$ are the strict transforms of the sections $\Delta_C, \{p_1\} \times C, \ldots , \{p_k\} \times C$.

 Using test curves of this form we
 can easily compute the coefficient $d$ of $\psi$ (respectively $\psi_1 + \psi_2$) in the expressions
 for $[\Hbar_{g,w}]$ and $[\Hbar_{g,g^1_2}]$ respectively.

\subsection{A special pencil of pointed hyperelliptic curves with non-trivial moduli}\label{sec.special}
Similar to the idea of varying an elliptic curve in an elliptic
surface, we can also vary a hyperelliptic curve in some special pencil
of hyperelliptic curves to obtain a family of pointed hyperelliptic
curves with at worst a single node.

If  $C \subset \P^1 \times \P^1$ is a smooth curve of bidegree $(2,g+1)$,
then by adjunction $g(C) = g$ and the projection $\pi_1 : C \to \P^1$
expresses $C$ as a double cover; \ie $C$ is hyperelliptic.
Let $\mathcal{C}\to \P^1$ 
be a general pencil of curves of bidegree $(2,g+1)$ in $\P^1 \times \P^1$; \ie
\[\mathcal{C} = Z(\lambda G(x,y,u,v) + \mu H(x,y,u,v) ) \subset
\P^1_{\lambda,\mu} \times \P^1_{x,y} \times \P^1_{u,v}\]
where both $G$ and $H$ are some fixed general bihomogenuous polynomials of bidegree $(2,g+1)$.

There are $4(g + 1)$ base points in this pencil corresponding
to the intersection of the two curves defined by $G$ and $H$ respectively
which define sections of the map ${\mathcal C} \to \P^1_{\lambda, \mu}$.
The projection to the second and third factors
${\mathcal C} \to \P^1 \times \P^1$ realizes ${\mathcal C}$ as the blowup
of $\P^1 \times \P^1$ at these $4(g+1)$ points with the sections being the
exceptional divisors.
Let $\sigma \colon \P^1 \to  {\mathcal C}$ be the section determined by one of
the base points.
\begin{lem} \label{lem.weierstrass}
  \[\deg_{\P^1}(\Hbar_{g,w}) =  1.\]
\end{lem}
\begin{proof}
  The degree of $\Hbar_{g,w}$ on this family can be calculated as
  $\deg \sigma^* (\tau^* O(\sigma))$ where $\tau$ is the hyperelliptic
  involution. The divisor $\sigma + \tau^*\sigma$
  is the pullback of a divisor under the projection to ${\mathcal C} \to \P^1_{x,y}$. Hence $(\sigma + \tau^*\sigma)^2 =0$. On the other hand, $\sigma^2 =-1$
  because $\sigma$ is an exceptional divisor of the blowup of $\P^1 \times \P^1$
  at the base point. Thus $(\tau^*\sigma)^2=-1$ as well, so
  we conclude that $\sigma \cdot \tau^*\sigma = 1$; \ie
  $\deg_{\P^1}(\Hbar_{g,w}) = 1$.
  \end{proof}

\begin{lem}\label{lem.nodal}
  The degree of $\eta_{irr}$ on this pencil is $4(2g+1)$ and the degrees of
  all other boundary divisors are zero.
  \end{lem}
\begin{proof}
  Since $\P^1 \times \P^1$ is a hypersurface in $\P^3$ the argument
  outlined in \cite[Section 7.4.2]{MR3617981}
  shows that a general pencil of curves of type $(a,b)$
  has a finite number of nodal fibers and each fiber has a single node
  which is not one of the base points. It follows
  that our pencil hyperellitpic curves embedded as curves of degree
  $(2,g+1)$ in $\Pro^1 \times \Pro^1$ intersects the boundary of $\Hbar_{g,1}$ only in $\eta_{irr}$.
  The divisor $\eta_{irr}$ on this family is the pullback from $\Mbar_{g,1}$ of
  the divisor $\delta_0$ because there are no fibers of type $\eta_i$ for $i > 0$. Given that only nodes appear in the singular fibers, the
  degree of this divisor is the degree of the discriminant of the linear
  systems $|{\mathcal O}(2,g+1)|$ on $\P^1 \times \P^1$. By \cite[Proposition 7.9]{MR3617981} this degree is $4(2g+1)$.
  \end{proof}
Using this special test curve with non-trivial moduli, we can easily get a $\Q$-linear relation between coefficients $d$ and $c$. We also get the similar statement for the $2$-pointed case by choosing two base points in the pencil. The detailed calculations will be given in Section \ref{formula-d} and Section \ref{2.relation-c-d}.

\subsection{A general construction of pencils of pointed hyperelliptic curves with at worst nodes as singularities}\label{sec.general}

To construct pencils of pointed hyperelliptic curves which have both non-trivial moduli and lie entirely in the boundary of $\Hbar_{g,1}$ or $\Hbar_{g,2}$, we want to attach a moving pointed family of hyperelliptic
curves of genus $h$ to a fixed pointed hyperelliptic curve of genus $g-h$
or $g-h-1$ depending on whether we want the pencil to lie in $\delta_h$
or $\eta_h$. For pencils lying in $\delta_h$ we need the attachment point to be a Weierstrass point, and for pencils lying in $\eta_h$ we need the two attachment points to sum to a $g^1_2$.

Unfortunately, the construction of Section \ref{sec.special} is not sufficiently flexible to allow us to make the section a Weierstrass section, or to mark two points which sum to a $g^1_2$. To rectify this problem we will
consider a variant on the construction of the pencil shown in Section \ref{sec.special} which allows more flexibility for creating families with sections which have certain properties. Similar constructions have been previously considered in
\cite[Section 13.8]{MR2807457}
and \cite[Section 3.2]{larson2019integral}.

Any smooth hyperelliptic curve $C$ of genus $h$
can be embedded in the weighted projective space $\mathbb{P}(1,1,h+1)$ and expressed by an equation 
\[z^2 = f(x,y) = f_0x^{2h + 2} + f_1x^{2h + 1}y + \cdots + f_{2h + 2}y^{2h + 2},\]
where the binary form $f(x,y)$ has distinct roots. Note that the weighted projective space $\Pro(1,1,h+1)$
is singular at the point $(0:0:1)$ but the image of $C$ misses this point.

Letting the coefficients $f_0, \ldots , f_{2h+2}$ vary in a
one-parameter family produces a pencil of
hyperelliptic curves embedded in $\Pro(1,1,h+1) \times \Pro^{2h + 2}$.
Since the discriminant locus of forms with
multiple roots is an ample divisor, any complete curve necessarily contains
singular curves, but we can arrange that the singular curves have a
single node, corresponding to forms that have a single double root and
all other roots simple.  The basic pencil of hyperelliptic curves we
consider is the pencil $X \subset \Pro(1,1,h+1) \times
\Pro^1_{\lambda, \mu}$, defined as
$$z^2 = f_{\lambda, \mu}(x,y) = \prod_{i = 1}^2[(a_i\lambda + b_i\mu)x + (c_i\lambda + d_i\mu)y]\prod_{j = 1}^{2h} (x + \alpha_j y)$$
where $a_i,b_i,c_i,d_i,\alpha_j$ are fixed constants in $k$.

\begin{lem}\label{tangency}
The pencil $X \to \Pro^1$ intersects the divisor $\eta_{irr}$ at $4h + 2$ points, each of multiplicity $2$. 
\end{lem}
\begin{proof}
It is clear that for general $a_i, b_i, c_i, d_i, \alpha_j$ there are $4h+2$
values of $(\lambda : \mu)$ for which the form $f_{\lambda, \mu}$ has a single
  double root.
  On the other hand the pencil we constructed induces morphism $\P^1 \to \Hbar_{g,1}$ factors through
  an embedding of $\Pro^1 \hookrightarrow \Pro^{2h+2}$ of degree $2$. 
of degree two. The pullback of $\eta_{irr}$ to $\Pro^{2h+2}$ is the discriminant which has degree $4h+2$. Thus, the intersection
of our pencil with the discriminant has degree $2(4h+2)$. Moreover, the local structure of the intersection does not depend on the double point, so each intersection point necessarily has the same multiplicity $2$.

Thus the family of hyperelliptic curves $X \to \Pro^1_{\lambda, \mu}$ has $4h+2$ fibers each with a single node and the intersection with $\eta_{irr}$ has multiplicity two at each point of $\Pro^1_{\lambda, \mu}$ with a singular fiber.
\end{proof}

The binary form $f_{\lambda, \mu}(x,y)$ defines a divisor of degree $(2,2h+2)$
in $\Pro^1_{\lambda, \mu} \times \Pro^1_{x,y}$. Since this divisor is a square
in $\Pic(\Pro^1 \times \Pro^1) = \Z \times \Z$, the equation
$z^2 = f_{\lambda, \mu}(x,y)$ expresses the family $X \to \Pro^1$ as a double
cover of $\Pro^1 \times \Pro^1$ branched along the divisor $B_0 = Z(f_{\lambda, \mu}(x,y))$.

The divisor $B_0$ decomposes as a sum of $2h+2$ irreducible components
$\Delta_1 + \Delta_2 + \sigma_1 + \ldots + \sigma_{2h}$ which 
gives sections of the projection
$\Pro^1_{\lambda, \mu} \times \Pro^1_{x,y} \to \Pro^1_{\lambda, \mu}$.
The sections $\Delta_{1}, \Delta_{2}$ are given by the formula
$$(\lambda: \mu) \mapsto \left( -(c_i  \lambda + d_i \mu): (a_i \lambda + b_i \mu)\right)$$ for $i = 1,2$ and have degree
$(1,1)$. The sections $\sigma_i$ are the constant sections defined by
$(\lambda: \mu) \mapsto (-\alpha_j:1)$.

Since the sections intersect transversally in $4h+2$ points, the total space
$X$ has $(4h+2)$ $A_1$-singularities at each of the nodes in the fibers
of the map $X \to \Pro^1_{\lambda, \mu}$. 

Let $P$ be the  blow up of the surface $\Pro^1_{\mu,\lambda} \times \Pro^1_{x,y}$ at
the $4h+2$ intersection points of the irreducible components of $B_0$.
The strict transforms $\widetilde{\Delta_1}, \widetilde{\Delta_2}, \widetilde{\sigma_1},
\ldots, \widetilde{\sigma_{2h}}$ of the sections $\Delta_1, \Delta_2, \sigma_1, \ldots
, \sigma_{2h}$ 
make $P \to \Pro^1_{\lambda, \mu}$ into a family of $(2h+2)$-pointed rational curves. Let $\pi: \widetilde{X} \to P$
be the double cover of $P$ branched along the divisor $\widetilde{B_0} = 
\widetilde{\Delta_1}+ \widetilde{\Delta_2}+ \widetilde{\sigma_1}+ 
\ldots+ \widetilde{\sigma_{2h}}$. Then
the data $(\pi: \widetilde{X} \to P, \widetilde{B_0})$ is an admissible covering
whose stabilization is our original family of hyperelliptic curves
$X \to \Pro^1_{\lambda, \mu}$.

Here are some intersection multiplicities that we will use quite often in later computations:
\[(\sigma_i \cdot \sigma_i)_{\Pro^1 \times \Pro^1} = 0, \quad
(\Delta_j \cdot \Delta_j)_{\Pro^1 \times \Pro^1} = 2.\]
Since each $\sigma_i$ gets blown up twice; each $\Delta_j$ gets blown up $(2h + 2)$ times and we have
\[(\widetilde{\sigma_i} \cdot \widetilde{\sigma_i})_{P} = -2, \quad (\widetilde{\Delta_j} \cdot \widetilde{\Delta_j})_{P} = -2h.\]
After taking the double cover $\widetilde{X}$, since all $\widetilde{\sigma}_i$'s and $\widetilde{\Delta_j}$'s are contained in the branch locus $\widetilde{B_0}$, we have
\[(s_i \cdot s_i)_{\widetilde{X}} = -1,\quad (t_j \cdot t_j)_{\widetilde{X}} = -h,\]
where $s_i = \pi^{-1}(\widetilde{\sigma_i})$ and $t_j = \pi^{-1}(\widetilde{\Delta_j})$, with $\pi^*(\widetilde{\sigma_i}) = 2s_i$ and $\pi^*(\widetilde{\Delta_j}) = 2t_j$.

Our main purpose for using this construction (see Figure \ref{fig.branchlocus}) is to obtain Weierstrass sections
of the family of hyperelliptic curves $f : \widetilde{X} \to \P^1$ which
can be used to attach another component or for marking a Weierstrass point. We will
use the following four ways to get sections.

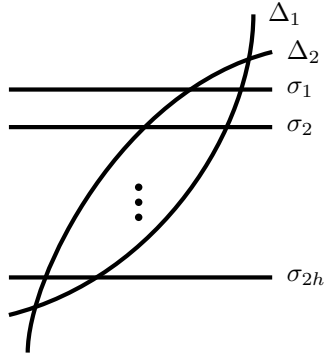
\begin{figure}\centering
\begin{tikzpicture}
\draw[name path = line1,ultra thick] (0,-.5) -- (3.5,-.5);
\node[right=2pt] at (3.5,-.5) {$\sigma_1$};
\draw[name path = line2,ultra thick] (0,-1) -- (3.5,-1);
\node[right=2pt] at (3.5,-1) {$\sigma_2$};
\node[fill,circle,inner sep=1pt] at (1.725,-1.8) {};
\node[fill,circle,inner sep=1pt] at (1.725,-2) {};
\node[fill,circle,inner sep=1pt] at (1.725,-2.2) {};
\draw[name path = line2i,ultra thick] (0,-3) -- (3.5,-3);
\node[right=2pt] at (3.5,-3) {$\sigma_{2h}$};

\draw[name path = diagonal1,ultra thick]
	(0,-3.5) .. controls (2,-3) and (3.25,-1) .. (3.25,0.5);
\node[right=2pt] at (3.25,0.5) {$\Delta_1$};

\draw[name path = diagonal2,ultra thick]
	(0.25,-4) .. controls (0.25,-3) and (1.5,-0.5) .. (3.5,0);
\node[right=2pt] at (3.5,0) {$\Delta_2$};
\end{tikzpicture}
\caption{The branch locus}
\label{fig.branchlocus}
\end{figure}

In order to get a Weierstrass section either for marking,
or for attaching to another component we can do the following:
\begin{enumerate}
\item  \label{sec.horiz} Choose the section $\tau \colon \P^1 \to \widetilde{X}$
  to be the inverse image of any of the sections $\widetilde{\sigma_i}$
  under the double cover map $\widetilde{X}$. By the above calculation
  $\tau^2 = -1$ in $\widetilde{X}$.
  
\item \label{sec.diag} Choose the section $\theta \colon \P^1 \to \widetilde{X}$
  to be the inverse image of either section
  $\widetilde{\Delta_i}$, which lies in the branch locus $\widetilde{B_0}$. By the above calculation
  $\theta^2 = -h$.
\end{enumerate}

In order to get two sections adding up to $g^1_2$ we use two constructions.
These families can be viewed as either being entirely contained in $\Hbar_{g,g^1_2}$ or they can be used for gluing to produce pencils contained in $\eta_{i}$
for some $i$.

\begin{enumerate}
\item \label{twosec.smooth} Start with a general horizontal ruling of $\Pro^1 \times \Pro^1 \to \Pro^1$. Since it intersects the branch locus at two distinct points, after taking the double cover, it becomes a double cover of $\mathbb{P}^1$ branched at two points. The family obtained by pullback along this double cover map is a family of hyperelliptic curves with two sections entirely contained in the divisor $\Hbar_{g,g^1_2}$,
  and each section is the inverse image of a smooth point of the fiber of $X \to \Pro^1$
  under the morphism which forgets the section.

\item \label{twosec.sing} Start with a horizontal ruling which intersects
  $\Delta_1$ and $\Delta_2$ at a common point. The inverse image of this ruling is a singular divisor with two irreducible components. After blowing up the point of intersection,
  we obtain a family with two sections, again entirely contained in $\Hbar_{g,g^1_2}$. 
\end{enumerate}
To get the final expressions in Theorem \ref{main-1} and \ref{main-2}, only a few of the above variants are needed.

\subsection{The test family $F_{2i+1}$}\label{sec.f2i+1}
Here we briefly recall the construction of $F_{2i + 1}$ in \cite[Section 13.8]{MR2807457} that produces a one-parameter family of hyperelliptic curves whose general fiber is smooth and whose singular fibers either have a node of type $\eta_{irr}$ or are in the
divisor $\eta_{i}$.

Start with $2i+2$ divisors in $|\mathcal{O}(1,1)|$ passing through a common point $p \in \Pro^1 \times \Pro^1$, and $2g - 2i$ divisors in $|\mathcal{O}(1,0)|$. After resolving the singularities in the branch locus we are allowed to take a double cover of the blown up surface $\widetilde{P}$ since the branch divisor is a square.
After stabilizing the double cover we will obtain a family with $(i+1)(4g -2i +1)$ fibers each with a single non-disconnecting node and one fiber of type
$\eta_i$. As was the case for our previous family, each fiber with an ordinary node contributes to $\eta_{irr}$ with multiplicity two. In particular,
we have
\begin{align*}
&\deg_{F_{2i+1}}(\eta_{i}) = 2, \\
&\deg_{F_{2i+1}}(\eta_{irr}) = 2(i+1)(4g - 2i + 1). 
\end{align*} 
  Taking the inverse images of the strict transforms to $\widetilde{P}$
  of other sections in $\Pro^1 \times \Pro^1$ gives a method of producing
  families with sections and specified fiber types. If the section
  is a component of the branch divisor, then we obtain a family with Weierstrass section;
  if the section is not contained in the branch divisor, we obtain a multi-section. After a degree $2$ base change to the section, we obtain a family with two sections which sum to a $g^1_2$.

\section{Computations}
In this section, we use the methods introduced in Section \ref{sec.methods} to prove Theorem \ref{main-1} and Theorem \ref{main-2}.

\subsection{Proof of Theorem \ref{main-1}}\label{sec.proof-1}
We will use the method of test curves to get express
$[\Hbar_{g,w}]$ in terms of a basis for $\Cl(\Hbar_{g,1}) \otimes \Q$ given in Theorem \ref{thm.scavia}.

\subsubsection{The first test curve -- calculating the coefficient of $\psi$
in \eqref{eq.hgw}}

\begin{prop}\label{formula-d}
  The coefficient $d$ of $\psi$ in equation \eqref{eq.hgw} equals $\frac{g+1}{g - 1}$.
\end{prop}
\begin{proof}
Let $C$ be a fixed smooth hyperelliptic curve and consider the family of pointed curves over $C$, defined by
$(C \times C \stackrel{\pi}\to C, \Delta)$
where $\Delta$ denotes the diagonal section. Since all fibers of this family are smooth
we know that all boundary divisors vanish on $C$.  Also $\deg_C ([\Hbar_{g,w}]) = 2g +2$
since the curve $C$ has $2g+2$ Weierstrass points. Since the section is the diagonal,
$\deg_C (\psi) = -(\Delta)^2 = 2g-2$.
Hence $2g+2 = (2g-2)d$ or equivalently, $d= \frac{g + 1}{g - 1}$.
\end{proof}

\subsubsection{Calculation of the coefficient of $\eta_{irr}$ in \eqref{eq.hgw}}
\begin{prop}\label{formula-c}
The coefficient $c$ of $\eta_{irr}$ in equation \eqref{eq.hgw} equals $\frac{-1}{2(2g + 1)(g - 1)}$.
\end{prop}
\begin{proof}
Let $(X \to \P^1, \sigma)$
be the pencil of hyperelliptic curves constructed in Section \ref{sec.special}. By Lemma \ref{lem.nodal}, we know the degree of the boundary class $\eta_{irr}$ on this family is
\[\deg_{\P^1}(\eta_{irr}) = 4(2g + 1).\]
Furthermore, since the marked point gets blown up once,
\[\deg_{\P^1}(\psi) = -(\sigma)^2 = 1.\]
Together with the fact that this family intersects other boundary divisors trivially, we obtain the following relation:
\[1 = 4(2g + 1)c + d,\quad\text{where }d = \dfrac{g + 1}{g - 1},\]
which implies that $c = \frac{-1}{2(2g + 1)(g - 1)}$ in \eqref{eq.hgw}.
\end{proof}

\subsubsection{Calculation the coefficients $a_{i,0}$, $a_{i,1}$, $b_{i,0}$ and $b_{i,1}$}

\begin{prop}\label{main.prop}
For $1 \leq i \leq \lfloor g/2 \rfloor$, 
the coefficients $a_{i,0}$ and $a_{i,1}$ of $\delta_{i,0}$ and $\delta_{i,1}$
in \eqref{eq.hgw} respectively, are
\begin{align}
&a_{i,0} = -\dfrac{2i(2i + 1)}{(2g + 1)(g - 1)}, \label{formula-a-0}\\
&a_{i,1} = -\dfrac{2(g - i)[2(g - i) + 1]}{(2g + 1)(g - 1)}.\label{formula-a-1}
\end{align}
For $1 \leq i \leq \lfloor (g-1)/2 \rfloor$,
the coefficients $b_{i,0}$ and $b_{i,1}$ of $\eta_{i,0}$ and $\eta_{i,1}$ respectively, are
\begin{align}
&b_{i,0} = -\dfrac{(2i + 1)(i + 1)}{(2g + 1)(g - 1)}, \label{formula-b-0}\\
&b_{i,1} = -\dfrac{[2(g - i) - 1](g-i)}{(2g + 1)(g - 1)}\label{formula-b-1}.
\end{align}
\end{prop}

\begin{proof}
  We will compute these coefficients using the construction introduced in Section \ref{sec.general} to produce test curves contained in the boundary
  divisor $\delta_{i,0}$ and $\delta_{i,1}$ for $1 \leq i \leq \lceil g/2 \rceil$which miss the Weierstrass divisor $\Hbar_{g,w}$.

\begin{enumerate}[label = (\roman*)]
\item Let $\widetilde{X} \to \P^1$ be a family of admissible covers of genus
  $i$ with Weierstrass section $\tau$ as in family \eqref{sec.horiz}.
  We form a family of pointed hyperelliptic curves contained in the divisor
  $\delta_{i,0}$ by attaching a fixed hyperelliptic curve of genus $g-i$
  at a Weierstrass point and choosing a general point for the marking.

In this family we have three types of fibers as illustrated in Figure \ref{fig.fibers-1}.
\begin{figure}\centering
\begin{tikzpicture}
\draw[name path=curve1,ultra thick] 
  (0.8,0) .. controls (0.8,-1) and (0.8,-2) .. (0,-3);
\draw[name path=curve2,ultra thick]
  (0,-2) .. controls (0.5,-3) .. (0.8,-4);
  
\path[name intersections={of=curve1 and curve2, by={a}}];
\node[fill,circle,inner sep=2pt] at (a) {};  
\node[right=2pt] at (a) {$\scriptstyle{w}$};

\path[name path=horizontal] (0,-1) -- (1,-1);
\path[name intersections={of=curve1 and horizontal, by = {b}}];
\node[fill,circle,inner sep=2pt] at (b) {};
\node[right=2pt] at (b) {$\scriptstyle{g-i}$};

\path[name path=horizontal1] (0,-3) -- (1,-3);
\path[name intersections={of=curve2 and horizontal1, by = {c}}];
\node[right=2pt] at (c) {$\scriptstyle{i}$};

\node at (0.4,-4.5) {general fibers};

\begin{scope}[xshift=3cm]
\draw[name path=curve3,ultra thick] 
  (0,0) .. controls (0,-1) and (0.3,-1.5) .. (1,-1.5);
\draw[name path=curve4,ultra thick]
  (1.2,-2) .. controls (0,-2) and (0,-3.5) .. (1.5,-3.5);
  
\node[right=2pt] at (1.5,-3.5) {$\scriptstyle{i-1}$};
\draw[name path=curve5,ultra thick] (0.5,-1) -- (1,-4);

\path[name intersections={of=curve3 and curve5, by={b}}];
\node[fill,circle,inner sep=2pt] at (b) {};
\node[left=2pt] at (b) {$\scriptstyle{w}$};

\path[name path=hori1] (0,-0.5) -- (1,-0.5);
\path[name intersections={of=curve3 and hori1, by={c}}];
\node[fill,circle,inner sep=2pt] at (c) {};
\node[right=2pt] at (c) {$\scriptstyle{g-i}$};

\node[right=2pt] at (0.5,-1) {$\scriptstyle{0}$};

\node at (0.75,-4.5) {$2$ fibers};
\end{scope}

\begin{scope}[xshift=6cm]
\draw[name path=curve6,ultra thick]
  (0,0) .. controls (0,-1) and (0.3,-1.5) .. (1,-1.5);
\draw[name path=curve7,ultra thick]
  (1,-1) .. controls (-0.5,-3) and (-0.5,-3.5) .. (1.5,-3.5);
\node[right=2pt] at (1.5,-3.5) {$\scriptstyle{i-1}$};
\draw[name path=curve8,ultra thick] (0,-1.5) -- (1,-4);

\path[name intersections={of=curve6 and curve7, by={e}}];
\node[fill,circle,inner sep=2pt] at (e) {};

\node[left=2pt] at (e) {$\scriptstyle{w}$};
\node[left=2pt] at (0,-1.5) {$\scriptstyle{0}$};

\path[name path=hori1] (0,-0.5) -- (1,-0.5);
\path[name intersections={of=curve6 and hori1, by={c}}];
\node[fill,circle,inner sep=2pt] at (c) {};
\node[right=2pt] at (c) {$\scriptstyle{g-i}$};

\node at (0.75,-4.5) {$4i$ fibers};
\end{scope}
\end{tikzpicture}
\caption{Fiber types}
\label{fig.fibers-1}
\end{figure}
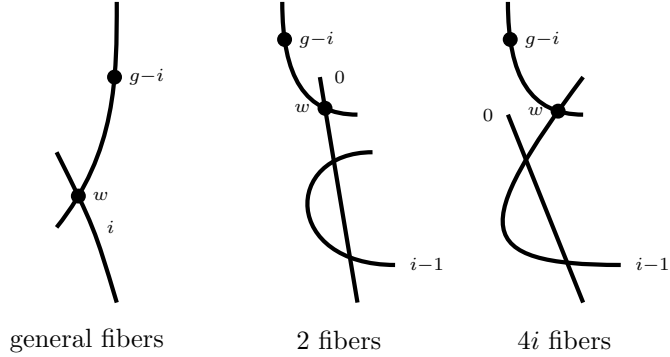

Note that the fibers of the middle type in Figure \ref{fig.fibers-1} are contained in the boundary divisors $\delta_{i,0}$ and $\eta_{i-1,0}$, but not in $\eta_{irr}$, since if they were contained in $\eta_{irr}$, they should appear when the Weierstrass point collides with a pair of points summing to the $g^1_2$. The only
possible way this occurs is when the pair of points summing to the $g^1_2$ collide
at a Weierstrass point. If this were the case the stable curve would necessarily have a component which is a
nodal elliptic curve. The image of our family $\Pro^1$ in $\Hbar_{g,1}$ misses
the intersection of the divisor $\eta_{i-1,0} \cap \eta_{i-1,1}$ so we may view
$\eta_{i-1,0}$ as a Cartier divisor and define its degree on our family. Moreover we claim that the divisor $\eta_{i-1,0}$ intersects this family $\Pro^1$ transversally.
To see this we argue as follows. The morphism 
$\Pro^1 \to \Hbar_{g,1}$ factors through the clutching morphism
$\Hbar_{i,w} \times \Hbar_{g-i,w} \to \delta_{i,1} \subset \Hbar_{g,1}$.
Let $D \subset \Hbar_{i,w}$ be divisor whose general point parametrizes a Weierstrass pointed nodal curve where the marking is on a rational bridge as illustrated in Figure \ref{fig.typeD}.
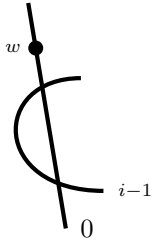
\begin{figure} 
	\begin{tikzpicture}
\draw[name path=curve4,ultra thick]
  (1.2,-2) .. controls (0,-2) and (0,-3.5) .. (1.5,-3.5);
  
\node[right=2pt] at (1.5,-3.5) {$\scriptstyle{i-1}$};
\draw[name path=curve5,ultra thick] (0.5,-1) -- (1,-4);

\node[fill,circle,inner sep=2pt] at (0.6,-1.6) {};
\node[left=2pt] at (0.6,-1.6) {$\scriptstyle{w}$};	
\node[right=2pt] at (1,-4) {$0$};	
  \end{tikzpicture}
  \caption{Fiber type D}
  \label{fig.typeD}
\end{figure}

The first factor of the map $\Pro^1 \to \Hbar_{g - i,w} \times \Hbar_{i,w}$ is constant and the fibers
of type $\eta_{i-1,0}$ are obtained by attaching a fixed curve in $\H_{g-i,w}$ to a moving curve in $\Hbar_{i,w}$.
It follows that the pullback
of $\eta_{i-1,0}$ to $\Pro^1$ equals the pullback of the divisor $D$ along the morphism $\Pro^1 \to \Hbar_{i,w}$
induced by the Weierstrass pointed family $\widetilde{X} \to \Pro^1$ \eqref{sec.horiz}.

To show that the intersection of our family with $D$ is transverse we need to understand the local structure
of $\Hbar_{i,w}$ in a neighborhood of a point of $D$. As noted in the proof of Proposition \ref{prop.irr}
the stack $\Hbar_{i,w}$ can be identified with the quotient stack $[\HH_{sm}(2i+2,w)/B]$  (see also \cite{EdHu:22}) where
$\HH_{sm}(2i+2,w)$ is the space of binary forms of degree $2i+2$ with distinct roots and which vanish at $(0:1)$.
The map $f \mapsto x_0f$ identifies $\HH_{sm}(2i+2,w)$ with $\Aff_{sm}(2i+1) \setminus L$
where $\Aff_{sm}(2i+1)$ is the space of binary forms of degree $2i+1$ with distinct roots and $L$ is the hyperplane parametrizing forms which vanish at
$(0:1)$. The quotient $[\Aff_{sm}(2i+1)/B]$ is a partial compactification
of $\H_{i,w}$ which includes the Weierstrass pointed curves of type $D$ and
the divisor $D$ is the quotient $[L/B]$.

For a suitable choice of coordinates on $\Pro^1$ the stabilization of the family
of admissible covers $\widetilde{X} \to \Pro^1$ \eqref{sec.horiz} is given by the equation
$$z^2 = x(\lambda x + \mu y) (\mu x + \lambda y) \prod_{i=4}^{2i+2} (a_ix + b_iy)$$
where $a_i,b_i$ are general constants and $(\lambda: \mu)$ are coordinates
on $\Pro^1$. With this choice of coordinates the map $\Pro^1 \to \Hbar_{i,w}$
factors through a map $\Pro^1 \to \Aff_{sm}(2i+1)$ given by
  $(\lambda: \mu) \mapsto (\lambda x + \mu y) (\mu x + \lambda y) \prod_{i=4}^{2i+2} (a_ix + b_iy)$. Direct inspection shows that this $\Pro^1$ intersects the divisor $L$
  transversally at the image of the points  $(\lambda: \mu) = (0:1)$ and $
  (\lambda: \mu) = (1:0)$ thereby demonstrating our transversality assertion.

Therefore
\begin{align}
&\deg_{\P^1}(\eta_{irr}) = 2(4i), \label{irr-tangency}\\
&\deg_{\P^1}(\delta_{i,0}) = \tau^2 = -1, \nonumber\\
&\deg_{\P^1}(\eta_{i-1,0}) = 2, \label{eta-i-notangency}
\end{align}
where the multiple $2$ in equation (\ref{irr-tangency}) is due to Lemma \ref{tangency}. 

Since the marked point is never a Weierstrass point we have
$\deg_{\P^1} [\Hbar_{g,w}] = 0$. Also, the section is constant so $\deg_B \psi =0$.
Thus we obtain the following relation
\begin{align}
0 = -a_{i,0} + 2(4i)c + 2b_{i-1,0}. \label{relation-a-b-1}
\end{align}
Switching the genera of these two components gives the following relation
\begin{align}
0 = -a_{i,1} + 2[4(g - i)]c + 2b_{i,1}. \label{relation-a-b-1'}
\end{align}

\item Now let $\widetilde{X} \to \P^1$ be a family of admissible covers of genus
  $i$ with Weierstrass section $\theta$ as in family \eqref{sec.diag}.
  Again, we form a family of pointed hyperelliptic curves contained in the divisor
  $\delta_{i,0}$ by attaching a fixed hyperelliptic curve of genus $g-i$
  at a Weierstrass point and choosing an arbitrary point for the marking.
  Again $\deg_{\P^1}[\Hbar_{g,w}] = 0$ and $\deg_{\P^1}\psi = 0$ and we have the
  following:  
\begin{align*}
&\deg_{\P^1}(\eta_{irr}) = 2(2i), \\
&\deg_{\P^1}(\delta_{i,0}) = \theta^2 = -i, \\
&\deg_{\P^1}(\eta_{i-1,0}) = 2i + 2.
\end{align*}
These give the relation:
\begin{align}
0 = 2(2i)c - ia_{i,0} + (2i + 2)b_{i-1,0}. \label{relation-a-b-2}
\end{align}
Similarly, by symmetry,
\begin{align}
0 = 2[2(g - i)]c - (g - i)a_{i,1} + [2(g - i) + 2]b_{i,1}. \label{relation-a-b-2'}
\end{align}
\end{enumerate}
Using these relations (\ref{relation-a-b-1}),(\ref{relation-a-b-1'}),(\ref{relation-a-b-2}) and (\ref{relation-a-b-2'}), we get the formula for  the coefficients of $\delta_{i,0},\delta_{i,1}$ and $\eta_{i,0},\eta_{i,1}$:
\begin{align*}
&a_{i,0} = -\dfrac{2i(2i + 1)}{(2g + 1)(g - 1)}, \\
&a_{i,1} = -\dfrac{2(g - i)[2(g - i) + 1]}{(2g + 1)(g - 1)}, \\
&b_{i,0} = -\dfrac{(i+1)(2i + 1)}{(2g + 1)(g - 1)}, \\
&b_{i,1} = -\dfrac{(g - i)[2(g - i) - 1]}{(2g + 1)(g - 1)}.
\end{align*}
which prove the proposition.
\end{proof}
Therefore, together with the results from the previous sections, we have proved Theorem \ref{main-1}.

\subsection{Proof of Theorem \ref{main-2}}\label{sec.proof-2}
The expression of $[\Hbar_{g,g^1_2}]$ will be computed in terms of a basis for $\Cl(\Hbar_{g,2}) \otimes \Q$ given in Theorem \ref{thm.scavia}. Recall that there is an extra boundary divisor $\delta_{0,2}$ in $\Hbar_{g,2}$.

\subsubsection{The first and second test curves -- calculating the coefficients
  $a_{0,2}$ and $d$ in \eqref{eq.hg12}}
\begin{prop}\label{2.relation-a-d}
  The coefficient $d$ of $\psi_1+\psi_2$ equals $1\over{g-1}$ and
  the coefficient $a_{0,2}$ of $\delta_{0,2}$
  equals $-{{g+1}\over{g-1}}$ in equation \eqref{eq.hg12}.
\end{prop}
\begin{proof}
  We will consider two families which intersect the boundary only in $\delta_{0,2}$ and use it to obtain two independent relations between $a_{0,2}$ and $d$.
  
  Let $C$ be a fixed smooth hyperelliptic curve of genus $g$, and let $p \in C$ be a non-Weierstrass point. Let $\widetilde{C \times C}$ be the blowup of $C$ at the point
  $(p,p)$ and let $\widetilde{\sigma}$ and $\widetilde{\Delta}$ be the strict transforms
  of  $C \times \{p\}$ and the diagonal respectively.
  The family $\widetilde{C \times C} \to C$ with the sections $\sigma_1 = \widetilde{\sigma}$,
and  $\sigma_2 = \widetilde{\Delta}$ is a family of two-pointed hyperelliptic curves.

  All fibers are non-singular except for the fiber over the point $p \in C$ which is in $\delta_{0,2}$.
  Thus, $\deg_C(\delta_{0,2}) =1$ and all other boundary divisors have degree 0. Since $p$ is not a Weierstrass point, there is a unique point $q \in C$
  such that $p+q$ is a $g^1_2$ hence $\deg_{C}(\Hbar_{g,g^1_2}) = 1$ by Proposition
  \ref{prop.hg12.normal}
  Also note that $\widetilde{\sigma}^2 = \sigma^2 - 1 = -1$ so $\deg_C(\psi_1) = 1$,
  and $\widetilde{\Delta}^2 = \Delta^2 - 1 = 1 - 2g$ so $\deg_C(\psi_2) = 2g-1$.
  Putting these together we obtain the relation

  \begin{equation} \label{relation-a-d}
    1 =  a_{0,2} + d + (2g - 1)d = a_{0,2}+ d(2g).
  \end{equation}

  Next let $\rho \colon \widetilde{C \times C} \to C \times C$ be the blowup of the $2g+2$ points $(p,p)$
  where $p$ is a Weierstrass point of $C$. Let $\sigma_1$ be the strict
  transform of the diagonal and $\sigma_2$ the strict transform
  of the section of $C \times C \to C$ given by $x \mapsto (x, \tau(x))$
  where $\tau$ is the hyperelliptic involution.

  Since the blowup map $\rho$ has $(2g+2)$ exceptional curves $E_1,\ldots, E_{2g+2}$
  each with self-intersection $-1$, we have
  that $\sigma_1^2 = \sigma_2^2 = 2-2g -(2g+2) = -4g$. Hence,
  $\deg_C (\psi_1) = \deg_C (\psi_2) = 4g$.
  
  The family $\widetilde{C \times C} \to C$ is contained entirely in $\Hbar_{g,g^1_2}$. Since
  $\sigma_2$ is obtained from $\sigma_1$ by the hyperelliptic involution,
  Proposition \ref{prop.hg12.normal} implies that  $\deg_C (\Hbar_{g,g^1_2}) = \sigma_1 \cdot (p^*\Delta) = \sigma_1 \cdot (\sigma_1 + E_1 + \ldots E_{2g+2}) = 2-2g$.

  Finally our family has $2g+2$ fibers of type $\delta_{0,2}$ so
  $\deg_C (\delta_{0,2}) = 2g+2$. This yields the relation
  \begin{equation}
    2-2g = (2g+2)a_{0,2} + 8gd.\label{2.relation-d-a}
  \end{equation}
  Combining \eqref{relation-a-d} and \eqref{2.relation-d-a}
  proves the proposition.
\end{proof}

\subsubsection{Calculating the coefficient of $\eta_{irr}$ in \eqref{eq.hg12}}
\begin{prop}\label{2.relation-c-d}
The coefficient $c$ for $\eta_{irr}$ in \eqref{eq.hg12} equals ${-1\over{2(g-1)(2g+1)}}$.
\end{prop}
\begin{proof}
We use the construction in Section \ref{sec.special} where
we take two of the $4(g + 1)$ base points which we denote by
$p_1,p_2$ to produce two sections $\sigma_1, \sigma_2$ of the pencil. Since the pencil ${\mathcal C} \to \P^1_{\lambda, \mu}$ is general, the two points $p_1,p_2$ are not contained in any ruling of the quadric $Q = \mathbb{P}^1 \times \mathbb{P}^1$. It means that no section of a $g^1_2$ contains both $p_1$ and $p_2$. Thus $\deg_{\P^1}([\Hbar_{g,g^1_2}]) = 0$.
We also have:
\begin{align*}
&\deg_{\P^1}(\eta_{irr}) = 4(2g + 1), \\
&\deg_{\P^1}(\psi_1) = -\widetilde{\sigma_1}^2 = 1, \\
&\deg_{\P^1}(\psi_2) = -\widetilde{\sigma_2}^2 = 1, 
\end{align*}
which gives us the relation
\begin{align}
0 = 4(2g + 1)c + 2d.
\end{align}
Since $d = {1\over{g-1}}$ we obtain the desired formula for $c$.
\end{proof}

\subsubsection{Calculation of the remaining coefficients in \eqref{eq.hg12}}
We use the same construction as in the proof of Proposition \ref{main.prop}
of families of hyperelliptic curves of genus $i$ (resp. genus $g-i$)
with a Weierstrass section
except that we attach a fixed two-pointed curve of genus $g-i$ (resp. genus
$i$) along the Weierstrass section.
Since we can ensure that the marked points on the fixed curve do not sum to a $g^1_2$ we can assume that our family misses $\Hbar_{g,g^1_2}$ and we obtain
the following relations.

When the Weierstrass section is obtained from a horizontal ruling we have:
\begin{align}
&0 = -a_{i,0} + 2(4i)c + 2b_{i-1,0}, \label{2.relation-a-b-1}\\
&0 = -a_{i,2} + 2[4(g - i)]c + 2b_{i,2}. \label{2.relation-a-b-1'}
\end{align}
When the Weierstrass section is obtained from a ruling of degree $(1,1)$
we get the following relations:
\begin{align}
&0 = -ia_{i,0} + 2(2i)c + 2(i + 1)b_{i-1,0}, \label{2.relation-a-b-2}\\
&0 = -(g - i)a_{i,2} + 2\cdot 2(g-i)c + 2(g - i + 1)b_{i,2} \label{2.relation-a-b-2'}.
\end{align}
By relations (\ref{2.relation-a-b-1}) and (\ref{2.relation-a-b-2}), together with Proposition \ref{2.relation-c-d}, we obtain
\begin{align}
&a_{i,0} = -\dfrac{2i(2i + 1)}{(g - 1)(2g + 1)}, \label{2.formula-a-0}\\
&b_{i,0} = -\dfrac{(i+1)(2i + 1)}{(g - 1)(2g + 1)}. \label{2.formula-b-0}
\end{align}
Likewise, by relations (\ref{2.relation-a-b-1'}) and (\ref{2.relation-a-b-2'}), together with Proposition \ref{2.relation-c-d}, we obtain
\begin{align}
&a_{i,2} = -\dfrac{2(g - i)[2(g - i) + 1]}{(g - 1)(2g + 1)} \label{2.formula-a-2}, \\
&b_{i,2} = -\dfrac{(g - i)[2(g - i) - 1]}{(g - 1)(2g + 1)} \label{2.formula-b-2}.
\end{align}

Now let's complete the calculation by finding the coefficients $a_{i,1}$ and $b_{i,1}$ of the classes $\delta_{i,1}$ and $\eta_{i,1}$ respectively.

\begin{prop}\label{2.formula-a-b-1}
For $1 \leq i \leq \lfloor g/2 \rfloor$, the coefficient $a_{i,1}$ of $\delta_{i,1}$ in \eqref{eq.hg12} is
\begin{align}
a_{i,1} = \dfrac{(2i-1)[2(g - i) - 1] - 2}{(g - 1)(2g + 2)}. \label{2.formula-a-1}
\end{align}
For $1 \leq i \leq \lfloor (g-1)/2 \rfloor$, the coefficient $b_{i,1}$ of $\eta_{i,1}$ in \eqref{eq.hg12} is
\begin{align}
b_{i,1} = \dfrac{2i(g - i - 1) - 1}{(g - 1)(2g + 1)}. \label{2.formula-b-1}
\end{align} 
\end{prop}
\begin{proof}
  We start with the basic construction introduced in Section
  \ref{sec.general}. We will take an extra general horizontal ruling
  $\sigma_0 \in |\mathcal{O}(1,0)|$ for the marked point on the genus
  $i$ component, and again take $\sigma_1$ which is contained in the
  branch locus to be the Weierstrass section for attachment. After 
  taking the double cover of $\Pro^1 \times \Pro^1$, the inverse image $s_0$
of $\sigma_0$ becomes a double cover of
  $\Pro^1$ branched at $2$ points. Thus, after a necessary degree $2$
  base change to $s_0$ we obtain a stable one-parameter family of
  hyperelliptic curves of genus $i$ over the base $\P^1$ which is the
  double cover of $\P^1$ branched at two points. This family is equipped
  with $2$ sections, one of which is
a Weierstrass section. Then we attach each curve in this
family to a fixed smooth genus $g - i$ at a Weierstrass point
and mark an additional arbitrary point on the component of genus $g-i$.
The result of this construction is a one-parameter family of stable genus $g$
hyperelliptic curves with two marked points -- one on the component of
genus $i$ and the other on the component of genus $g-i$. This family has empty 
  intersection with the divisor $\Hbar_{g,g^1_2}$ and we deduce the 
  following relation among $c,d$ and $a_{i,1},b_{i-1,1}$. The computation of the degrees is similar to that of Proposition \ref{main.prop}.
\begin{align}
-2a_{i,1} + 4b_{i-1,1} + 2\cdot 2(4i)c + 2d = 0.
\end{align} 
By switching the genera, we get the symmetric relation
\begin{align}
-2a_{i,1} + 4b_{i,1} + 2\cdot 2 \cdot 4(g - i)c + 2d = 0.
\end{align}
There is a second family of this form where we use the diagonal section of
the Weierstrass divisor for gluing. Unfortunately, this section will hit
the node of one special fiber producing after blowup a family with a rational
bridge. Since we cannot ensure that this family is contained in the smooth
locus of $\Hbar_{g,2}$, we cannot compute the degrees of various divisors on the family.

To remedy this problem we use the family $F_{2i+1}$ of \cite[Section 13.8]{MR2807457} which was
described in Section \ref{sec.f2i+1}. To obtain a family with two sections
we consider the inverse image of a general section $\sigma_0 \in |\mathcal{O}(1,0)|$ and a special section $\sigma_0' \in |\mathcal{O}(1,0)|$ passing the common point where $2i+2$ divisors in the linear system $|{\mathcal O}(1,1)|$
meet.
After resolving the singularities in the branch locus and then
taking the double cover, the inverse image of $\sigma_0$, denoted by $s_0$, is now a double cover of $\Pro^1$ branched at $2i + 2$ points and the inverse image of $\sigma_0'$ now gets separated into
two disjoint lines passing through the exceptional divisor, either one of which is denoted by $s_0'$. It can be checked that $s_0' \cdot s_0' = -1$.

After a degree $2$ base change to $s_0$, we obtain a one-parameter family with base
$B \coloneq s_0$ of pointed genus $g$ hyperelliptic curves with sections.
Notice that $s_0$ is a smooth genus $i$ curve, and its inverse image under the base change consists of two components, which we denote by $s_0^{(1)}$ and $s_0^{(2)}$, satisfying $s_0^{(1)} \cdot s_0^{(2)} = 2i + 2$. By an easy computation on $\left(s_0^{(1)} + s_0^{(2)}\right)^2 = 2s_0^2 = 0$, we get $s_0^{(j)} \cdot s_0^{(j)} = -(2i + 2)$ for $j = 1,2$.
Choosing one section to be either $s_0^{(1)}$ or $s_0^{(2)}$,
and one section to be the inverse image of $s_0'$, we obtain a two-pointed family
of stable hyperelliptic curves contained in the smooth locus of
$\Hbar_{g,2}$ which has empty intersection with $\Hbar_{g,g^1_2}$. On this
family, together with the results from \cite[Section 13.8]{MR2807457} stated in Section \ref{sec.f2i+1}, we have
\begin{align*}
&\deg_B(\psi_1) = 2(i+1), \\
&\deg_B(\psi_2) = 2, \\
&\deg_B(\eta_{i,1}) = 2, \\
&\deg_B(\eta_{irr}) = 2\cdot 2(i+1)(4g - 2i + 1),
\end{align*}
which yields the relation
\begin{align}
b_{i,1} + 2(i+1)(4g - 2i + 1)c + (i + 2)d = 0.
\end{align}
Together with the facts from Proposition \ref{2.relation-c-d} and \ref{2.relation-a-d}, we get
\begin{align*}
b_{i,1} = \dfrac{2i(g - i - 1) - 1}{(g - 1)(2g + 1)}. 
\end{align*}
Using either one of relations mentioned above involving $a_{i,1}$, we get
\begin{align*}
a_{i,1} = \dfrac{(2i-1)[2(g - i) - 1] - 2}{(g - 1)(2g + 2)},
\end{align*}
which finishes the proof.
\end{proof}
Combining the results from Proposition \ref{2.relation-c-d}, \ref{2.relation-a-d} and \ref{2.formula-a-b-1}, with \eqref{2.formula-a-0},\eqref{2.formula-a-2},\eqref{2.formula-b-0}, and \eqref{2.formula-b-2}, we complete the proof of Theorem \ref{main-2}.

\bibliographystyle{alpha}
\bibliography{Hyperelliptic.sns.revised}

\end{document}